\documentclass[]{amsart}
\usepackage{amssymb}
\usepackage{amsthm}
\usepackage{amsmath}
\usepackage{bbm}
\usepackage{enumitem}
\usepackage{lmodern}
\usepackage[usenames,dvipsnames]{color}
\usepackage{hyperref}
\hypersetup{
 colorlinks,
 linkcolor={red!50!black},
 citecolor={blue!50!black},
 urlcolor={blue!80!black}
}
\usepackage[a4paper, inner=3cm,outer=3cm,top=3.3cm,bottom=3cm]{geometry}
\usepackage{subcaption}
\usepackage{graphicx}
\usepackage{pgfplots}
\usepackage{tikz,tikz-3dplot}
\usetikzlibrary{cd}
\usetikzlibrary{shapes.misc}
\usetikzlibrary{arrows,decorations.markings}
\usepackage{xifthen}

\AtBeginDocument{%
\def\MR#1{}
}
\theoremstyle{definition}
\newtheorem{thm}{Theorem}
\newtheorem{prop}[thm]{Proposition}
\newtheorem{cor}[thm]{Corollary}
\newtheorem{lemma}[thm]{Lemma}
\newtheorem{conj}[thm]{Conjecture}

\theoremstyle{definition}

\newtheorem{ex}[thm]{Example}
\newtheorem{rmk}[thm]{Remark}

\newcommand{\supp}{{\mathrm{supp}}}

\newcommand{\sqp}[1][]{%
\ifthenelse{\isempty{#1}}{\mathcal{P}(\square_2)}{\mathcal{P}^{#1}(\square_2)}%
}

\newcommand{\Q}{{\mathbb{Q}}}

\newcommand{\Z}{{\mathbb{Z}}}

\newcommand{\rleft}{\mathopen{}\mathclose\bgroup\left}
\newcommand{\rright}{\aftergroup\egroup\right}

\DeclareMathOperator{\chr}{char}

\def\coloneqq{\mathrel{\mathop:}=}
\title{When are symmetric ideals monomial?}
\author[A.\,Kretschmer]{Andreas Kretschmer}
\address[A.\,Kretschmer]{Institut f\"ur Algebra und Geometrie\\Fakult\"at f\"ur Mathematik\\Otto-von-Guericke-Universit\"at Magdeburg}
\curraddr{}
\email{andreas.kretschmer@ovgu.de}
\thanks{}

\begin{document}
\maketitle{}
\setlength{\parindent}{0pt}
\nocite{*}

\begin{abstract}
	We study conditions on polynomials such that the ideal generated by their orbits under the symmetric group action becomes a monomial ideal or has a monomial radical. If the polynomials are homogeneous, we expect that such an ideal has a monomial radical if their coefficients are sufficiently general with respect to their supports. We prove this for instance in the case where some generator contains a power of a variable. Moreover, if the polynomials have only square-free terms and their coefficients do not sum to zero, then in a larger polynomial ring the ideal itself is square-free monomial. This has implications also for symmetric ideals of the infinite polynomial ring.
\end{abstract}

\section{Introduction and results}

Over the last years there has been considerable interest in symmetric closed subschemes of both finite and infinite-dimensional affine spaces defined by ideals of polynomial rings which are invariant under the action of the symmetric group, see e.g. \cite{Draisma2014Noetherianity} for the infinite case. A key structural result in the infinite case is that these ideals are generated by the orbits of finitely many polynomials \cite{Cohen1967Laws,Aschenbrenner2007Finite,Hillar2012Finite}. In~\cite{Nagel2017Equivariant}, Hilbert functions for invariant chains of ideals are introduced and shown to be rational functions. The Hilbert functions of symmetric \emph{monomial} ideals in $K[x_1, x_2, \ldots]$ were investigated in~\cite{Nagel2018Equivariant}. The Betti numbers of symmetric monomial ideals in a polynomial ring with finitely many variables are studied in~\cite{Murai2020Equivariant}, where at the end of the introduction the authors ask for an extension of their results to the case of generators which are not necessarily monomials.

This short note asks how much of a restriction the monomial case actually is, and it is meant to bring more awareness to the fact that solution sets to symmetric systems of polynomials are often very simple even if the individual polynomials are not. Theorem~\ref{thm:main} and Conjecture~\ref{conjecture:supp_length_k} are attempts at making this intuition precise in the finite setting, and Theorem~\ref{thm:squarefree_terms} is an indication for how the infinite case might actually yield simpler and more explicit results. For a focus on Specht polynomials instead of monomials, see the recent~\cite{Moustrou21Symmetric}.

\subsection*{Notation}
Let $K$ be a field. The set of occurring monomials in a polynomial $f \in K[x_1, \ldots, x_n]$ is called its \emph{support} and denoted $\supp(f)$. We identify a monomial in $K[x_1, \ldots, x_n]$ with its exponent vector in $(\Z_{\geq 0})^n$ and call any non-empty, finite subset $\mathcal{A} \subseteq (\Z_{\geq 0})^n$ a \emph{support set}. The support set $\mathcal{A}$ is called \emph{homogeneous of degree $d$} if $\sum_{i=1}^n a_i = d$ for all $a \in \mathcal{A}$. We identify $K^\mathcal{A}$ with the set of all polynomials $f \in K[x_1, \ldots, x_n]$ with $\supp(f) \subseteq \mathcal{A}$.
We write $S_n$ for the symmetric group and $S_\infty$ for the (small) infinite symmetric group, acting on polynomial rings by permuting the variables. We denote the action of $\sigma \in S_n$ on a polynomial $f$ by $\sigma.f$, and it is induced by letting $\sigma.x_i \coloneqq x_{\sigma(i)}$. For the entire orbit of $f$ under the action of a subgroup $G \subseteq S_n$ we write $G.f$.
Listing the exponents of a monomial in decreasing order gives a partition of the degree of the monomial, and we call this partition the \emph{type} of the monomial. Two monomials have the same type if and only if they are permutations of one another.
Whenever we use the term \emph{general}, the field $K$ is assumed to be infinite. Fixing a support set $\mathcal{A}$, an assertion about polynomials $f \in K[x_1, \ldots, x_n]$ with $\supp(f) \subseteq \mathcal{A}$ holds for \emph{general coefficients of $f$ with respect to $\mathcal{A}$} if the subset of $K^\mathcal{A}$ for which the assertion holds contains a non-empty Zariski-open subset.
Given a set of polynomials $S$, by $\mathcal{V}(S)$ we denote its vanishing set.

\begin{thm}\label{thm:main}
	Let $\mathcal{A} \subseteq (\Z_{\geq 0})^n$ be a homogeneous support set and $f \in K^\mathcal{A}$. Denote by $k$ the minimal number of strictly positive entries among all elements of $\mathcal{A}$ and let $G \subseteq S_n$ be a subgroup.
	\begin{enumerate}
		\item\label{item:homogeneous} Assume $\mathcal{A}$ contains a power of some variable and let $G$ act transitively on the variables. Then, for general coefficients of~$f$ with respect to $\mathcal{A}$, $\sqrt{(G.f)} = (x_1, \ldots, x_n)$ is the irrelevant ideal, in particular $\mathcal{V}(G.f) = \{0\}$.
		\item\label{item:one_type_monomial} Let all monomials in $\mathcal{A}$ be of the same type and let $G$ act transitively on the set of all monomials of this type. Then, for general coefficients of~$f$ with respect to $\mathcal{A}$, $(G.f)$ is monomial, generated by the orbit of any term of~$f$.
		\item\label{item:symmetric_supp} Assume $\chr(K) = 0$ and $n \geq 5$. Let $\mathcal{A}$ be symmetric, i.e., every permutation of an element of $\mathcal{A}$ lies in $\mathcal{A}$ as well. Then, for general coefficients of~$f$ with respect to $\mathcal{A}$, $\sqrt{(S_n.f)} = (S_n. x_1 x_2 \cdots x_k)$. In particular, $\mathcal{V}(S_n.f)$ consists of all elements of $K^n$ with at least $n-k+1$ zero entries.
	\end{enumerate}
\end{thm}

Theorem~\ref{thm:main} shows that the surprising behavior of \cite[Example~2.6]{Juhnke2020Asymptotic} is not rare. We even expect the following to hold true.

\begin{conj}\label{conjecture:supp_length_k}
	Let $\mathcal{A} \subseteq (\Z_{\geq 0})^n$ be a support set and $f \in K^\mathcal{A}$. Denote by $k$ the minimal number of strictly positive entries among all elements of $\mathcal{A}$. Then, for general coefficients of~$f$ with respect to~$\mathcal{A}$, $\sqrt{(S_n.f)} = (S_n. x_1 x_2 \cdots x_k)$ if $\mathcal{A}$ is homogeneous, and
	\begin{equation*}
		\mathcal{V}(S_n.f) \subseteq \mathcal{V}(S_n. x_1 x_2 \cdots x_k) \cup \mathcal{V}(x_i^e - x_j^e: i,j = 1, \ldots, n)
	\end{equation*}
	for some $1 \leq e \leq \deg(f)$ if $\mathcal{A}$ is inhomogeneous.
\end{conj}

Two immediate consequences of Theorem~\ref{thm:main}\eqref{item:homogeneous} in geometric terms are the following.

\begin{cor}\label{cor:main}
	Let $Z \subseteq \mathbb{A}^n$ be a symmetric homogeneous subscheme, i.e., $Z$ is defined by the $S_n$-orbits of homogeneous polynomials $f_1, \ldots, f_r \in K[x_1, \ldots, x_n]$. If $f_1$ contains a power of a variable, then for sufficiently general coefficients of $f_1$ with respect to its support, set-theoretically $Z = \{0\} \subseteq \mathbb{A}^n$, and the corresponding subscheme of projective space is empty. Similarly for a symmetric homogeneous subscheme of the infinite affine space.
\end{cor}

\begin{cor}\label{cor:main2}
	Let $K$ be infinite and $G \subseteq S_{n+1}$ a subgroup acting transitively on the variables $x_0, \ldots, x_n$. Let $X \subseteq \mathbb{P}^n_K$ be a general degree $d$ hypersurface, viewed as a $K$-rational point in $\mathbb{P}(K[x_0, \ldots, x_n]_d)$. Then
	\begin{equation*}
		\bigcap_{\sigma \in G} \sigma(X)  = \emptyset.
	\end{equation*}
\end{cor}

One possible choice for $G$ is $\Z/(n+1)\Z$, acting by cyclically permuting the variables in any given order. In this case, Corollary~\ref{cor:main2} states that the $n+1$ cyclic permutations of a general homogeneous polynomial $f$ forms a homogeneous system of parameters for $K[x_0, \ldots, x_n]$.

\begin{rmk}
	In view of Theorem~\ref{thm:main}, one could be lead to think that the ideal generated by the orbits of homogeneous generators with sufficiently general coefficients should itself be monomial. This, however, is not true. A simple counterexample is given in Example~\ref{ex:counter}. Moreover, Theorem~\ref{thm:main} suggests that symmetric ideals should rarely be expected to be radical, see however Theorem~\ref{thm:squarefree_terms} for a notable exception.
\end{rmk}

\begin{rmk}\label{rmk:inhomogeneous}
	If $f$ is inhomogeneous, then $(S_n.f)$ usually does not contain any monomial. Indeed, if $f$ has at least two homogeneous parts $f_i$ and $f_j$ which do not vanish at $(1,1,\ldots,1)$, then $f(t,t,\ldots,t)$ is an inhomogeneous univariate polynomial which therefore has a non-zero solution in $\overline{K}$, in particular $\mathcal{V}_{\overline{K}}(S_n.f)$ has a torus solution.
\end{rmk}

\begin{thm}\label{thm:squarefree_terms}
	Assume $\chr(K) = 0$ or $\chr(K) > n$. Let $f \in K[x_1, \ldots, x_n]$ be a homogeneous polynomial of degree $d$ having only square-free terms. If $f(1, 1, \ldots, 1) \neq 0$ in $K$, then for all $N \geq n+d$ in the polynomial ring $K[x_1, \ldots, x_N]$ we have
	\begin{equation*}
		(S_N.f) = (S_N.x_1 x_2 \cdots x_d).
	\end{equation*}
	Conversely, if $f(1, 1, \ldots, 1) = 0$ in $K$, then $\sqrt{(S_N.f)}$ does not contain any monomial for any $N \geq n$.
\end{thm}

Theorem~\ref{thm:squarefree_terms} can also be viewed as a representation-theoretic statement about the $S_N$-representation with a basis given by all $d$-element subsets of $\{1, \ldots, N\}$. The character of this representation is known explicitly\footnote{see for example \href{https://mathoverflow.net/questions/123721/permutation-character-of-the-symmetric-group-on-subsets-of-certain-size}{this} mathoverflow post}, nonetheless the result does not seem to follow in a straightforward way. The proof given below is purely combinatorial.

\begin{rmk}
	If $\chr(K) = 0$, Theorem~\ref{thm:squarefree_terms} implies that the $K$-linear $S_\infty$-representation $V_d$ given by all $d$-element subsets of the natural numbers $\mathbb{N}$ has a unique maximal proper subrepresentation, namely the subvector space of $V_d$ defined by all coefficients summing to zero. Indeed, the theorem implies that any element which does not lie in this subvector space generates all of $V_d$. We expect this to be known to experts on the representation theory of the infinite symmetric group.
\end{rmk}

\begin{rmk}
	Theorem~\ref{thm:squarefree_terms} is a first hint that the statements of Theorem~\ref{thm:main} and Conjecture~\ref{conjecture:supp_length_k} might become much nicer, with the possibility of obtaining the genericity conditions explicitly, in \emph{large enough} polynomial rings, i.e., if the ideals are considered in polynomial rings with sufficiently many variables, or even in the infinite polynomial ring. This feature also appears in \cite{Moustrou21Symmetric}.
\end{rmk}

\section{Proofs}

\begin{proof}[Proof of Theorem~\ref{thm:main}(\ref{item:homogeneous})]
	We first work over the algebraic closure $\overline{K}$. Let $V^\ast$ be the $\overline{K}$-vector space of all polynomials with support lying inside $\mathcal{A}$. By assumption, there exists $i$ with $x_i^d \in \mathcal{A}$. Now, the image under the first projection $\mathbb{P}(V^\ast) \times \mathbb{P}^{n-1} \overset{\mathrm{pr}_1}{\longrightarrow} \mathbb{P}(V^\ast)$ of the closed subset
	\begin{equation*}
		X \coloneqq \{([f],[x]): f(\sigma.x) = 0 \text{ for all } \sigma \in G\} \subseteq \mathbb{P}(V^\ast) \times \mathbb{P}^{n-1}
	\end{equation*}
	is closed in $\mathbb{P}(V^\ast)$ because projective space $\mathbb{P}^{n-1}$ is complete. By construction, the complement of $\mathrm{pr}_1(X)$ in $\mathbb{P}(V^\ast)$ is precisely the set of homogeneous polynomials $f$, up to scaling, with support contained in $\mathcal{A}$ such that $\sqrt{(G.f)} = (x_1, \ldots, x_n)$. This set is Zariski-open, so it is enough to see that it is non-empty. But clearly, $[x_i^d]$ is contained.
	
	Secondly, we deduce the claim for arbitrary infinite fields $K$. Let $N \coloneqq |\supp(f)|$. By the above, there is a non-empty principal open subset $D_{\overline{K}}(\alpha) \subseteq \mathbb{A}^N_{\overline{K}}$ for which the assertion of the theorem holds. But $\alpha \in \overline{K}[c_1, \ldots c_N]$ only has finitely many coefficients, so there actually is a finite field extension $L$ of $K$ such that $\alpha \in L[c_1, \ldots c_N]$. The integral ring extension $L[c_1, \ldots, c_N] \hookrightarrow \overline{K}[c_1, \ldots, c_N]$ induces the surjective morphism of affine schemes $b: \mathbb{A}^N_{\overline{K}} \rightarrow \mathbb{A}^N_L$, and as $b^{-1}(D_L(\alpha)) = D_{\overline{K}}(\alpha)$, we obtain $b(D_{\overline{K}}(\alpha)) = D_L(\alpha)$. The finite ring extension $K[c_1, \ldots, c_N] \hookrightarrow L[c_1, \ldots c_N]$ induces the finite surjective morphism $b': \mathbb{A}^N_L \rightarrow \mathbb{A}^N_K$ of finite type $K$-schemes, to which Chevalley's theorem on constructible subsets applies. The dimensions of $\mathbb{A}^N_L$ and $\mathbb{A}^N_K$ agree, and by finiteness of $b'$ the constructible image of the open dense subset $D_L(\alpha)$ under $b'$ is necessarily dense in $\mathbb{A}^N_K$, hence contains a non-empty principal open $D_K(\beta) \subseteq \mathbb{A}^N_K$ (where now $\beta \in K[c_1, \ldots, c_N]$). Restricting to the set of $K$-rational points, $D_K(\beta) \cap K^N$ is still a non-empty open of the irreducible space $K^N$ (with the subspace topology from $\mathbb{A}^N_K$). Both follows from $K$ being infinite as in this case there is no non-zero polynomial in $K[c_1, \ldots, c_N]$ vanishing on all of $K^N$.
	
	Finally, let $f \in D_K(\beta) \cap K^N$. Then $f$ is identified with its corresponding polynomial $f \in K[x_1, \ldots, x_n]$ having $\supp(f) \subseteq \mathcal{A}$. As $f \in D_K(\beta)$, we have $\sqrt{(G.f)} = (x_1, \ldots, x_n)$ in the polynomial ring over $\overline{K}$. So, interpreting all ideals now in the polynomial ring over~$K$, we get $(x_1, \ldots, x_n)/\sqrt{(G.f)} \otimes_{K[x_1, \ldots, x_n]} \overline{K}[x_1, \ldots, x_n] = 0$. But since $\overline{K}[x_1, \ldots, x_n]$ is a free module over $K[x_1, \ldots, x_n]$, the equality $\sqrt{(G.f)} = (x_1, \ldots, x_n)$ also follows over~$K$.
\end{proof}

\begin{proof}[Proof of Theorem~\ref{thm:main}(\ref{item:one_type_monomial})]
	Denote by $\overline{\sigma}$ the coset of the permutation $\sigma \in G \subseteq S_n$ in the set $Q \coloneqq G/H$ where $H$ is the stabilizer of the fixed monomial $x_1^{d_1} \cdots x_n^{d_n}$ in $\mathcal{A}$. Write $f = \sum_{\overline{\sigma} \in Q} c_{\overline{\sigma}} x_{\sigma(1)}^{d_1} \cdots x_{\sigma(n)}^{d_n}$ such that $c_{\overline{\mathrm{id}}} \neq 0$. We again identify $f$ with its coefficient vector $c_f = (c_{\overline{\sigma}})_{\overline{\sigma} \in Q}$ in the vector subspace $V^\ast$ of $\mathrm{Sym}^d((K^n)^\ast)$ generated by all monomials of the same type as $x_1^{d_1} \cdots x_n^{d_n}$. In this way, $V^\ast$ is a linear representation of $G \subseteq S_n$ via the restriction of the diagonal $\mathrm{GL}_n$-action on $\mathrm{Sym}^d((K^n)^\ast)$. The monomial $x_1^{d_1} \cdots x_n^{d_n}$ lies in the ideal $(G.f)$ if and only if its coefficient vector $e_1$ lies in the span of all $c_{\tau.f} \in V^\ast$ with $\tau \in G$. Putting these vectors as columns of a $\dim(V^\ast) \times |G|$-matrix $C = (c_{\tau.f})_{\tau \in G}$, this is equivalent to $e_1$ lying in the image of $C$. As the $G$-orbit of $e_1$ generates all of $V^\ast$, this in turn is equivalent to $C$ having full rank $\dim(V^\ast)$. This clearly defines a Zariski-open subset of the affine space $K^\mathcal{A}$ of coefficient vectors of polynomials with support contained in $\mathcal{A}$. This open set is non-empty as $x_1^{d_1} \cdots x_n^{d_n}$ itself is clearly contained.
\end{proof}

\begin{proof}[Proof of Theorem~\ref{thm:main}(\ref{item:symmetric_supp})]
	Let $\chr(K) = 0$ and $n \geq 5$. There are precisely two $S_n$-representations of dimension $1$, the trivial representation and the sign representation. Every irreducible $S_n$-representation of dimension $>1$ has dimension at least $n-1$ for $n \geq 5$. This follows from the classical representation theory of the symmetric groups, see for example \cite{Fulton1991Representation}. If $m_0$ is a monomial, the \emph{permutation module} $M^{m_0}$ is the $K$-vector space generated by all permutations of $m_0$, with the obvious structure of an $S_n$-representation. The multiplicities of the irreducible $S_n$-representations, i.e. the Specht modules, inside $M^{m_0}$ are classically known as the \emph{Kostka numbers}. We will use that the trivial representation always has multiplicity exactly $1$ in $M^{m_0}$ and is spanned by the monomial symmetric polynomial inside $M^{m_0}$ and the sign representation has multiplicity $0$ if not all exponents of $x_1, \ldots, x_n$ in $m_0$ are distinct. If they are all distinct, then the sign representation has multiplicity $1$ in $M^{m_0}$ as well and is spanned by the polynomial which has coefficient $1$ in front of all $A_n$-permutations of $m_0$ and coefficient $-1$ in front of all the others. In particular, the sum of the isotypic components of the trivial and the sign representation in $M^{m_0}$ has the property that for each of its elements all coefficients of the $A_n$-permutations of $m_0$ agree.
		
	We can assume that $k \leq n-1$, otherwise we can divide $f$ by the appropriate power of $x_1 x_2 \cdots x_n$ and proceed with the resulting polynomial. Using Theorem~\ref{thm:main}\eqref{item:homogeneous} and~\eqref{item:one_type_monomial} we can assume that $\mathcal{A}$ contains at least two monomials of different types and no power of a variable. With the same argument as in the proof of Theorem~\ref{thm:main}\eqref{item:homogeneous}, we can reduce to the case $K = \overline{K}$. Let $V$ be the vector space with $\mathcal{A}$ as a basis. We write the elements of $V$ as vectors $(y_m)_{m \in \mathcal{A}}$. Its dual $V^\ast \cong K^\mathcal{A}$ is interpreted as the space of polynomial functions with support contained in $\mathcal{A}$. We write elements of $V^\ast$ as coefficient vectors $(c_m)_{m \in \mathcal{A}}$. The action of $S_n$ on $\mathcal{A}$ induces compatible linear actions on $V$ and $V^\ast$. Consider now the rational map $\varphi \colon \mathbb{P}^{n-1} \dashrightarrow \mathbb{P}(V)$, sending $[x_1 : \cdots : x_n]$ to the homogeneous coordinate vector of all monomials in $\mathcal{A}$ evaluated at $(x_1, \ldots, x_n)$. The indeterminacy locus of $\varphi$ is precisely $\mathcal{V}(S_n. x_1 x_2 \cdots x_k) \subseteq \mathbb{P}^{n-1}$. We denote its complement by $U$ and observe that $\varphi$ is $S_n$-equivariant. Consider the constructible subset $X \subseteq \mathbb{P}(V^\ast) \times \mathbb{P}(V)$ given by
	\begin{equation*}
	X = \{([c_m],[y_m]): [y_m] \in \varphi(U), \sum_{m \in \mathcal{A}} c_m y_{\sigma.m} = 0 \text{ for all } \sigma \in S_n\}.
	\end{equation*}
	Denoting by $\mathrm{pr}_1$, $\mathrm{pr}_2$ the projections onto the factors, the constructible set $\mathrm{pr}_1(X) \subseteq \mathbb{P}(V^\ast)$ is precisely the set of polynomials $g$, up to scaling, for which the assertion of the theorem does not hold, i.e., $\sqrt{(S_n.g)} \neq (S_n. x_1 x_2 \cdots x_k)$. Hence, it is enough to see that the dimension of $X$, and thus of $\mathrm{pr}_1(X)$, is at most $\dim(\mathbb{P}(V^\ast)) - 1$. To prove the latter, we will write $X$ as a union of finitely many constructible subsets which all satisfy this dimension bound. Namely, we distinguish three classes of points $[(y_m)_{m \in \mathcal{A}}] \in \varphi(U) \subseteq \mathbb{P}(V)$ according to properties of the $S_n$-subrepresentation of $V$ generated by all permutations of $(y_m)_{m \in \mathcal{A}}$.
	\begin{enumerate}[label=(\arabic*)]
		\item First, assume there is some $m_0 \in \mathcal{A}$ which is not a power of $x_1 x_2 \cdots x_n$ and such that $0 \neq y_{m_0} = y_{\sigma.m_0}$ for all $\sigma \in A_n$. There are in fact only finitely many such points in $\varphi(U)$. In order to see this, let $Z \subseteq \mathbb{P}^{n-1}$ be the preimage under $\varphi$ of the set of all such points $[(y_m)_{m \in \mathcal{A}}] \in \varphi(U) \subseteq \mathbb{P}(V)$ (for fixed $m_0$). Then $Z$ is contained in the intersection of $\mathcal{V}(m_0 - \sigma.m_0: \sigma \in A_n) \subseteq \mathbb{P}^{n-1}$ with the algebraic torus. Write $m_0 = x_1^{e_1} \cdots x_n^{e_n}$. Without loss of generality assume $e_1 \geq e_2 \geq e_3$ and $e_1 > e_3$. Then $m_0 - (1,2,3).m_0$ is a monomial multiple of $m' \coloneqq x_1^{e_1 - e_3} - x_2^{e_1 - e_2} x_3^{e_2 - e_3}$. Hence, $Z$ is contained in the subset of the algebraic torus defined by all $A_n$-permutations of $m'$. Let $a \coloneqq e_1 - e_2$ and $b \coloneqq e_2 - e_3$. From the two equations $x_1^{a+b} = x_2^a x_3^b$ and $x_3^{a+b} = x_1^a x_2^b$ we deduce
		\begin{equation*}
		x_1^{(a+b)^2} = x_2^{a(a+b)} (x_3^{a+b})^b = x_2^{a(a+b)}(x_1^a x_2^b)^b = x_2^{a^2 + b^2 + ab} x_1^{ab},
		\end{equation*}
		and hence $x_1^{a^2 + b^2 + ab} = x_2^{a^2 + b^2 + ab}$. Therefore, $Z$ is contained in the closed subset of the algebraic torus defined by $x_1^e = x_2^e = \cdots = x_n^e$ with $e = (e_1 - e_3)^2 - (e_1 - e_2)(e_2 - e_3)$. Clearly, there are only finitely many solutions to these equations in $\mathbb{P}^{n-1}$, so the image of $Z$ under $\varphi$ is also a finite set of points. The fiber of $X \overset{\mathrm{pr}_2}{\longrightarrow} \mathbb{P}(V)$ over each of these points is contained in some hyperplane inside $\mathbb{P}(V^\ast)$.
		\item Second, assume that there is some $m_0 \in \mathcal{A}$ such that $m_0$ is not a power of $x_1 x_2 \cdots x_n$ and $\sum_{m \in S_n.m_0} y_m \neq 0$ but not all $y_m$ with $m$ in the $A_n$-orbit of $m_0$ agree. Then the fiber of $X \overset{\mathrm{pr}_2}{\longrightarrow} \mathbb{P}(V)$ over such a point is a linear space of codimension at least $n$ in $\mathbb{P}(V^\ast)$. Indeed, this fiber is the projectivization of the kernel of the matrix $Y$ whose rows are all $S_n$-permutations of $(y_m)_{m \in \mathcal{A}}$, and the codimension of this kernel in $V^\ast$ is precisely the rank of $Y$. We now claim that even the submatrix $Y'$ of $Y$ whose rows are all the $S_n$-permutations of $(y_m)_{m \in S_n.m_0}$ has rank at least $n$. This is because the $S_n$-representation given by the span of all $S_n$-permutations of $(y_m)_{m \in S_n.m_0}$ is a subrepresentation of the permutation module $M^{m_0}$ which contains the trivial representation (just sum all the $S_n$-permutations of $(y_m)_{m \in S_n.m_0}$) and also some other irreducible representation which is neither the trivial nor the sign representation. The last statement follows from the assumption as every vector $(y_m)_{m \in S_n.m_0}$ lying in the sum of the isotypic components of the trivial and the sign representation inside $M^{m_0}$ has equal entries at all places corresponding to the $A_n$-permutations of $m_0$, as noted in the beginning of the proof. Then the dimension of the preimage under $X \overset{\mathrm{pr}_2}{\longrightarrow} \mathbb{P}(V)$ of the set of all such points $[(y_m)_{m \in \mathcal{A}}] \in \varphi(U)$ is at most $\dim(\mathbb{P}(V^\ast)) - 1$.
		\item If the two cases above do not apply, then for all $m_0 \in \mathcal{A}$ which is not a power of $x_1 x_2 \cdots x_n$ we have $\sum_{m \in S_n.m_0} y_m = 0$ but not all $y_m$ with $m$ in the $A_n$-orbit of $m_0$ agree. This translates into the codimension $1$ condition $\sum_{m \in S_n.m_0} m(x_1, \ldots, x_n) = 0$ on $\mathbb{P}^{n-1}$. With a similar argument as in the previous case, the fiber of any such point $[(y_m)_{m \in \mathcal{A}}] \in \varphi(U)$ has codimension at least $n-1$, making the dimensions add up to at most $\dim(\mathbb{P}(V^\ast)) - 1$.\qedhere
	\end{enumerate}
\end{proof}

For smaller $n$ or symmetric \emph{inhomogeneous} support sets $\mathcal{A}$, the above proof can be adapted in many cases but the details are tedious, so I did not include them here.

Theorem~\ref{thm:squarefree_terms} is a consequence of the following more special result. In order to state it, denote by $e_n^d(x_1, \dots, x_n)$ the elementary symmetric polynomial of degree $d$ in $n$ variables, $1 \leq d \leq n$.

\begin{prop}\label{prop:elementary_symmetric}
	Let $I = (S_N.e_n^d(x_1, \ldots, x_n)) \subseteq K[x_1, \ldots, x_N]$ and $N \geq n+d$.
	\begin{enumerate}
		\item\label{radical} If $\chr(K) \nmid \binom{n}{d}$, then
		\begin{equation*}
			\sqrt{I} = (S_N.x_1 \cdots x_d).
		\end{equation*}
		Moreover, if $\chr(K) \mid \binom{n}{d}$, then for any $N \geq n$, the ideal $\sqrt{I}$ does not contain any monomial.
		\item\label{elimination} If $\chr(K) = 0$ or $\chr(K) > n$, then
		\begin{equation*}
			I = (S_N.x_1 \cdots x_d).
		\end{equation*}
	\end{enumerate}
\end{prop}

\begin{proof}[Proof of Proposition~\ref{prop:elementary_symmetric}(\ref{radical})]
		For the first claim, we start by showing that the polynomial $f \coloneqq (x_1 - x_{n+1}) \cdots (x_d - x_{n+d})$ lies in the ideal $I$ and hence so do all its permutations. Indeed, write
	\begin{equation*}
		f_1 \coloneqq e_n^d(x_1, \dots, x_n) - (1, n+1).e_n^d(x_1, \dots, x_n) = (x_1 - x_{n+1}) e_{n-1}^{d-1}(x_2, \dots, x_n) \in I.
	\end{equation*}
	Similarly, we see that
	\begin{equation*}
		f_2 \coloneqq f_1 - (2,n+2).f_1 = (x_1 - x_{n+1})(x_2 - x_{n+2})e_{n-2}^{d-2}(x_3, \dots, x_n) \in I.
	\end{equation*}
	Inductively, we obtain $f \in I$, as desired. On the level of vanishing sets, the inclusion $\sigma.f \in I$ for all $\sigma \in S_N$ implies that for any $y \in \mathcal{V}(I) \subseteq K^N$ we have $f(\sigma.y) = 0$ for every permutation $\sigma$ of the vector $y$. Now, every non-zero entry of $y$ occurs at most $n-1$ times. Indeed, if this entry $t \in K$ occurs at the indices $i_1 < \cdots < i_n$, then $0 = e_n^d(y_{i_1}, \ldots, y_{i_n}) = \binom{n}{d} t^d$, hence $t = 0$ by our assumption on the characteristic of $K$. But then it even follows that every non-zero entry can occur at most $d-1$ times. If namely the entry $t \neq 0$ occurs $r$ times with $d \leq r \leq n-1$, then there are at least $N-r \geq d+1$ entries of $y$ which are different from $t$ because $N \geq n+d$. This, however, contradicts $f(\sigma.y) = 0$ for all $\sigma$ as we can find a permutation $\sigma$ after which $y_1 = \cdots = y_d = t$ and $y_{n+1}, \ldots, y_{n+d}$ are all different from $t$. Therefore, given any $y \in \mathcal{V}(I)$, there is a permutation of the entries of $y$ after which they are assembled in a way such that
	\begin{align*}
		t_1 &\coloneqq y_1 = \cdots = y_{r_1} \neq 0, \\
		t_2 &\coloneqq y_{r_1 + 1} = \cdots = y_{r_1 + r_2} \neq 0, \\
		&\ \ \vdots \\
		t_s &\coloneqq y_{r_1 + \ldots + r_{s-1} + 1} = \cdots = y_{r_1 + \ldots + r_s} \neq 0
	\end{align*}
	and $y_{r_1 + \ldots + r_s+1} = \cdots = y_N = 0$, where $d > r_1 \geq r_2 \geq \cdots \geq r_s \geq 1$. We now claim that $r_1 + \ldots + r_s < d$, so there are at most $d-1$ non-zero entries for any vector $y \in \mathcal{V}(I)$. Assume to the contrary $r_1 + \ldots + r_s \geq d$. Let $\tau \in S_N$ be the permutation that interchanges $d+i$ with $n+i$ for all $1 \leq i \leq d$ and is the identity otherwise. Then $\tau.f = (x_1 - x_{d+1}) \cdots (x_d - x_{2d})$ and hence
	\begin{equation*}
		(\tau.f)(y_1, \ldots, y_N) = (\tau.f)(\underbrace{t_1, t_1, \ldots, t_j}_{d}, \underbrace{t_{j'}, \ldots, t_{j''}}_{d}, \underbrace{\ast, \ldots, \ast}_{N-2d}) \neq 0,
	\end{equation*}
	where $1 < j \leq j' < j''$. By assumption on the order of the $r_i$, no factor of $\tau.f$ can vanish, hence $\tau.f$ does not vanish at $y$, a contradiction. Therefore,
	\begin{equation*}
		\mathcal{V}(I) = \mathcal{V}(S_N. x_1 \cdots x_d).
	\end{equation*}
	We obtain $\sqrt{I} = (S_N.x_1 \cdots x_d)$. For the second claim we observe $e_n^d(1,1,\ldots,1) = \binom{n}{d} = 0$ in~$K$, so the all ones vector $(1,1,\ldots,1) \in K^N$ lies in $\mathcal{V}(I)$.
\end{proof}

\begin{lemma}\label{lemma:binomialIdentity}
	We have the identity
	\begin{equation*}
	\binom{n-1}{d} \sum_{j = 0}^d (-1)^j \frac{\binom{d-a}{j} \binom{n-d+a}{d-j}}{\binom{n-1}{d-j}} = \begin{cases} \binom{n}{d} & \text{for } a=d \\ 0 & \text{for all } 0 \leq a \leq d-1 \end{cases}
	\end{equation*}
	for all $1 \leq d \leq n-1$.
\end{lemma}

\begin{proof}
	For $a = d$ this can be checked immediately, and in the other cases this is well known if written equivalently in the form
	\begin{equation*}
	\sum_{j = 0}^r (-1)^j \binom{r}{j} (s+j)^{\underline{r-1}} = 0,
	\end{equation*}
	where $(s+j)^{\underline{r-1}}$ denotes the falling factorial and $r \coloneqq d-a \geq 1$, $s \coloneqq n-1-d \geq 0$. To see this, observe that this sum is just the $r$-th discrete derivative of the polynomial $s^{\underline{r-1}}$ which is of degree $r-1$ in $s$. Here, the discrete derivative of a polynomial $f(s)$ is defined as $(\Delta f)(s) = f(s+1) - f(s)$. Clearly, $\deg(\Delta f) \leq \deg(f) - 1$. It follows that $\Delta^{(r)}s^{\underline{r-1}} = 0$.
\end{proof}

\begin{proof}[Proof of Proposition~\ref{prop:elementary_symmetric}(\ref{elimination})]
	For a given field $K$, our goal is to find, if possible, a $K$-linear combination of the polynomials $\sigma.e_n^d$ with $\sigma \in S_{n+d}$ which equals a non-zero $K$-multiple of the monomial $x_1 \cdots x_d$. Let first $K = \Q$. We want to find coefficients $c_j \in \Q$ such that
	\begin{equation}\label{eq:eliminationIdentity}
		\binom{n}{d} x_1 \cdots x_d = \sum_{j = 0}^d (-1)^j c_j \sum_{|J_1| = d-j} e_n^d(x_J), 
	\end{equation}
	where the second sum ranges over all subsets $J \subseteq \{1, \ldots, n+d\}$ of cardinality $n$ such that $J_1 \coloneqq J \cap \{1, \ldots, d\}$ is of the given cardinality. Moreover, by $e_n^d(x_J)$ we denote the elementary symmetric polynomial of degree $d$ in the $n$ variables indexed by $J$. It is easy to see that only the summand with $j=0$ contributes elementary symmetric polynomials that contain the monomial $x_1 \cdots x_d$, and there are precisely $\binom{n}{d}$ of those, which forces $c_0 = 1$ for equation \eqref{eq:eliminationIdentity} to hold. Note moreover that in the $j$-th summand of \eqref{eq:eliminationIdentity}, all occuring $e_n^d(x_J)$ only contain monomials containing at most $d-j$ of the variables $x_1, \ldots, x_d$. More precisely, given a square-free monomial $x_A x_B$ with $A \subseteq \{1, \ldots d\}$, $B \subseteq \{d+1, \ldots, d+n\}$ of degree $d$ we write $a = |A|$ and $|B| = d-a$. Then in the $j$-th summand of the sum in \eqref{eq:eliminationIdentity}, the monomial $x_A x_B$ occurs exactly $\binom{d-a}{j} \binom{n-d+a}{d-j}$ times as a counting argument shows. Obviously, this does not depend on the sets $A$ and $B$ but only on the cardinality of $A$. Now define $c_1$ in a way such that the monomials $x_A x_B$ with $a = d-1$ in the $j=1$ summand cancel with the corresponding terms in the $j=0$ summand. Clearly, there exists a unique such $c_1 \in \Q$, and the $j=1$ summand does not contribute monomials with $a=d$. Similarily, define then $c_2 \in \Q$ to be the unique rational number such that the monomials $x_A x_B$ with $a = d-2$ in the $j=2$ summand cancel out all the corresponding terms in the $j=0$ and $j=1$ summands. Again, the $j=2$ summand cannot contribute any monomials with $a \geq d-1$. Continuing in this way, we define unique numbers $c_0, \ldots, c_d \in \Q$ depending only on $n$ and $d$ with $c_0 = 1$ and such that \eqref{eq:eliminationIdentity} must hold by construction. Now fix a monomial $x_A x_B$. Then the fact that this monomial has coefficient $\binom{n}{d}$ if $a = d$ and coefficient $0$ otherwise on the right hand side of the equation \eqref{eq:eliminationIdentity} precisely translates into the binomial identity
	\begin{equation*}
		\sum_{j = 0}^d (-1)^j c_j \binom{d-a}{j} \binom{n-d+a}{d-j} = \begin{cases} \binom{n}{d} & \text{for } a=d \\ 0 & \text{for all } 0 \leq a \leq d-1 \end{cases}
	\end{equation*}
	for all $1 \leq d \leq n-1$. From the uniqueness of the $c_j$ and Lemma~\ref{lemma:binomialIdentity} we then obtain $c_j = \frac{\binom{n-1}{d}}{\binom{n-1}{d-j}}$. Next, as we now know that the binomial identity of Lemma~\ref{lemma:binomialIdentity} holds over every field $K$ in which all expressions are defined, if additionally $\chr(K) \nmid \binom{n}{d}$, then the monomial $x_1 \cdots x_d$ lies in the ideal $I$ by \eqref{eq:eliminationIdentity}. Equivalently, a sufficient assumption on $\chr(K)$ for the elimination in \eqref{eq:eliminationIdentity} to work is that $\chr(K)$ does not divide $\binom{n}{d}$ nor any denominator of the reduced fractions $c_j$ for all $0 \leq j \leq d$. In particular, $\chr(K) = 0$ or $\chr(K) > n$ will suffice.
\end{proof}

\begin{proof}[Proof of Theorem~\ref{thm:squarefree_terms}]
	Let $\mathcal{J}$ be the set of all $d$-element subsets of $\{1, \ldots, n\}$ corresponding to terms of $f$, so that there are $c_J \in K$ for all $J \in \mathcal{J}$ such that $f = \sum_{J \in \mathcal{J}} c_J x_J$. We write $c \coloneqq f(1, 1, \ldots, 1) = \sum_{J \in \mathcal{J}} c_J$. The following polynomial clearly lies in the ideal $(S_N.f)$:
	\begin{equation*}
		\sum_{\sigma \in S_n} \sigma.f = \sum_{J \in \mathcal{J}} c_J \sum_{\sigma \in S_n} \sigma.x_J = \left(\sum_{J \in \mathcal{J}} c_J \right) d!(n-d)! e_n^d(x_1, \ldots, x_n) = c \cdot d! (n-d)! e_n^d(x_1, \ldots, x_n).
	\end{equation*}
	If $c \neq 0$ in $K$, it follows that $e_n^d(x_1, \ldots, x_n) \in (S_N.f)$ and so $(S_N.f) = (S_N.x_1 x_2 \cdots x_d)$ by Proposition~\ref{prop:elementary_symmetric}\eqref{elimination}. Conversely, if $c = 0$ in $K$, then the all ones vector $(1, 1, \ldots, 1) \in K^N$ lies in the vanishing set $\mathcal{V}(S_N.f)$, so $(S_N.f)$ and its radical do not contain any monomial. 
\end{proof}

\begin{ex}\label{ex:counter}
	Let $f \coloneqq x_1^2 + t x_1 x_2$. Then for all $t \neq 0$ the ideal $(S_n.f)$ does not contain $x_1^2$ for any $n \geq 2$. Suppose differently and let $c_\sigma \in K$ for $\sigma \in S_n$ such that $$x_1^2 = \sum_{\sigma \in S_n} c_{\sigma} (\sigma.f) = \sum_{\sigma \in S_n} c_\sigma x_{\sigma(1)}^2 + t \sum_{\sigma \in S_n} c_\sigma x_{\sigma(1)} x_{\sigma(2)}.$$ Then $\sum_{\sigma(1) = 1} c_{\sigma} = 1$ and $\sum_{\sigma(1) \neq 1} c_{\sigma} = 0$ by looking at the sum only involving squared variables. The sum only involving square-free variables becomes $$0 = t x_1 \sum_{\sigma(1) = 1} c_{\sigma} x_{\sigma(2)} + t \sum_{\sigma(1) \neq 1} c_{\sigma} x_{\sigma(1)} x_{\sigma(2)}.$$ Now, we set all variables $x_n = \cdots = x_3 \coloneqq x_2$ (but $x_1$ stays unchanged). Then the first sum becomes simply $t x_1 x_2$, and the second sum, after splitting it up as $$t \sum_{\sigma(1) \neq 1 \neq \sigma(2)} c_{\sigma} x_{\sigma(1)} x_{\sigma(2)} + t \sum_{\sigma(2) = 1} c_{\sigma} x_{\sigma(1)} x_{\sigma(2)},$$ becomes $$t x_2^2 (\sum_{\sigma(1) \neq 1 \neq \sigma(2)} c_{\sigma}) + t x_1 x_2 \sum_{\sigma(2) = 1} c_{\sigma}.$$ From $t \neq 0$ it follows that $\sum_{\sigma(1) \neq 1 \neq \sigma(2)} c_{\sigma} = 0$ and $\sum_{\sigma(2) = 1} c_{\sigma} = -1$. But then, $$-1 = \sum_{\sigma(1) \neq 1 \neq \sigma(2)} c_{\sigma} + \sum_{\sigma(2) = 1} c_{\sigma} = \sum_{\sigma(1) \neq 1} c_{\sigma} = 0,$$ a contradiction.
\end{ex}

\begin{ex}
	The genericity assumptions of Theorem~\ref{thm:main}\eqref{item:homogeneous} and \eqref{item:one_type_monomial} are necessary. Consider $f \coloneqq x_1^2x_2 + x_1x_2^2$ and $I \coloneqq (S_N.f) \subseteq \Q[x_1, \ldots, x_N]$. A computation in \texttt{Macaulay2} for $N = 3$ shows
	\begin{equation*}
		\sqrt{I} = (S_N.f, S_N.x_1x_2x_3),
	\end{equation*}
	so for all $N \geq 3$ we have $x_1 x_2 x_3 \in \sqrt{I}$. However, no permutation of $x_1 x_2$ lies in $\sqrt{I}$ because all permutations of $(1, -1, 0, 0, \ldots, 0)$ lie in $\mathcal{V}(S_N.f)$. Hence, all monomials in $I$ and $\sqrt{I}$ are divisible by at least $3$ distinct variables.
\end{ex}

\begin{ex}
	It is possible even for the ideal generated by the orbit of an inhomogeneous polynomial to be monomial although it is a rare phenomenon as explained by Remark~\ref{rmk:inhomogeneous}. An example is given by $I \coloneqq (S_3.(x_1 + x_2 + x_1^2 - x_2^2)) \in K[x_1, x_2, x_3]$ for any field $K$ of $\chr(K) \neq 2$. Indeed, one has $I = (x_1, x_2, x_3)$ as follows from
	\begin{equation*}
		2x_1 = (x_1 + x_2 + x_1^2 - x_2^2) + (x_3 + x_1 + x_3^2 - x_1^2) - (x_3 + x_2 + x_3^2 - x_2^2) \in I.
	\end{equation*}
\end{ex}

\begin{ex}
	Theorem~\ref{thm:squarefree_terms} and Proposition~\ref{prop:elementary_symmetric} are false in general for $N < n+d$. Consider the case $n=3$, $d=2$, $N = 4 < n+d$ for the ideal $I = (S_N.e_3^2) \subseteq \mathbb{Q}[x_1, x_2, x_3, x_4]$. Then a computation in \texttt{Macaulay2} gives the Gröbner basis
	\begin{equation*}
		\{x_1 x_2 - x_3 x_4, x_1 x_3 - x_2 x_4, x_1 x_4 + x_2 x_4 + x_3 x_4, x_2 x_3 + x_2 x_4 + x_3 x_4, x_2^2 x_4, x_2 x_4^2,x_3^2 x_4, x_3 x_4^2\}
	\end{equation*}
	with respect to the lexicographic monomial order, and no monomial of degree $2$ lies in $I$. In particular, $I$ is not radical.
\end{ex}

\begin{ex}
	The statement of Proposition~\ref{prop:elementary_symmetric}\eqref{elimination} is false in general if only $\chr(K) \nmid \binom{n}{d}$ holds. As an example, consider the case $K = \Z/2\Z$ with $n=3$, $d=2$ and $N = n+d = 5$ for the polynomial $e_3^2 = x_1 x_2 + x_1 x_3 + x_2 x_3$. Then $(x_1x_2)^2 \in I = (S_N.e_3^2)$ but $x_1x_2 \not\in I$, as can be checked by computing a Gröbner basis of $I$. The same is still true for $N=6,7$.
\end{ex}

\subsection*{Acknowledgments}
I want to thank Uwe Nagel and Martina Juhnke-Kubitzke for their lectures at the REACT workshop as well as the organizers of the latter. Proposition~\ref{prop:elementary_symmetric} is the result of a first try towards an open problem posed there. I am grateful to Jan Draisma for helpful comments concerning Theorem~\ref{thm:squarefree_terms}, to Satoshi Murai for pointing out \cite{Moustrou21Symmetric} which I was not aware of before, and to Benjamin Nill for explaining a quick proof of Lemma~\ref{lemma:binomialIdentity}. Moreover, thanks go to Thomas Kahle, Abeer Al Ahmadieh and Arne Lien for pointers to the literature, discussions and encouragement. Finally, \texttt{Macaulay2} has been of great help in gaining intuition by computing many examples. The author is supported by the Deutsche Forschungsgemeinschaft (DFG, German Research Foundation) -- 314838170, GRK~2297 MathCoRe.

\bibliographystyle{amsplain}
\bibliography{Sym-orbits.bib}
\end{document}